\documentclass{amsart}
\usepackage{amssymb}

\begin{document}

\newtheorem{theorem}{Theorem}[section]
\newtheorem{proposition}[theorem]{Proposition}
\newtheorem{lemma}[theorem]{Lemma}
\newtheorem{corollary}[theorem]{Corollary}
\newtheorem{conjecture}[theorem]{Conjecture}
\newtheorem{question}[theorem]{Question}
\newtheorem{claim}[theorem]{Claim}
\newtheorem{exercise}[theorem]{Exercise}

\theoremstyle{definition}
\newtheorem{definition}[theorem]{Definition}
\newtheorem{note}[theorem]{Notation}

\theoremstyle{remark}
\newtheorem{remark}[theorem]{Remark}
\newtheorem{example}[theorem]{Example}
\newtheorem{problem}[theorem]{Problem}

\def\square{\hfill${\vcenter{\vbox{\hrule height.4pt \hbox{\vrule width.4pt
height7pt \kern7pt \vrule width.4pt} \hrule height.4pt}}}$}

\def\R{\mathbb R}
\def\C{\mathbb C}
\def\Z{\mathbb Z}
\def\H{\mathbb H}
\def\1{\pmb{1}}
\def\length{\textnormal{length}}
\def\cone{\textnormal{cone}}
\def\level{\textnormal{level}}
\def\parent{\textnormal{parent}}
\def\comp{\textnormal{comp}}
\def\wake{\textnormal{wake}}
\def\Aut{\textnormal{Aut}}

\newenvironment{pf}{{\it Proof:}\quad}{\square \vskip 12pt}
\newenvironment{spf}{{\it Sketch of proof:}\quad}{\square \vskip 12pt}

\newcommand{\marginal}[1]{\marginpar{\tiny #1}}

\title{Counting subgraphs in hyperbolic graphs with symmetry}

\author{Danny Calegari, Koji Fujiwara 
}
\thanks{}
\address{Danny Calegari.
Department of Mathematics,
University of Chicago,
Chicago, IL 60637, USA.
{\tt dannyc@math.uchicago.edu}
}
\address{Koji Fujiwara.
Department of Mathematics,
Kyoto University,
Kyoto, 606-8502, Japan.
{\tt kfujiwara@math.kyoto-u.ac.jp}
}

\maketitle

\section{Introduction}
This note addresses some questions that arise in the series of works
by Kyoji Saito on the growth functions of graphs \cite{Sa}, \cite{Sa2}.
We study ``hyperbolike'' graphs, which include Cayley graphs of
hyperbolic groups. We generalize some well-known results on hyperbolic
groups to the hyperbolike setting (Theorem \ref{theorem:rationality},
Theorem \ref{theorem:exponential}), including rationality of generating
functions, and sharp estimates on the growth rate of vertices.
We then apply these results to confirm a conjecture of Saito on the 
``opposite series'', which was originally posed for hyperbolic groups
(Theorem \ref{main}, Corollary \ref{hyp}).
We also give a (standard) example of a hyperbolike graph with positive density of
dead ends, and point out its implications for the applicability of the 
main theorems in \cite{Sa}.

\subsection{Acknowledgements}

We would like to thank Laurent Bartholdi, Pierre-Emmanuel Caprace, 
Markus Pfeiffer, Kyoji Saito and Yasushi Yamashita
for helpful comments and discussions. Danny Calegari was partly supported by NSF grant DMS 1358592.
Koji Fujiwara was partly supported by a Grant-in-Aid for Scientific
Research No. 23244005.

\section{Generating functions of hyperbolike graphs}\label{section:definition_statement}

\begin{definition}[hyperbolike graph]
A connected graph $X$ of finite valence is {\em $\delta$-hyperbolike} 
for some $\delta\ge 0$ if it satisfies the following properties:
\begin{enumerate}
\item{$X$ is $\delta$-hyperbolic; and}
\item{$\Aut(X)$ is transitive on the vertices.}
\end{enumerate}
\end{definition}

Condition (2) implies that all vertices have the same valence --- i.e.\/ $X$ is
regular. Moreover, by hypothesis, this (common) valence is finite. Thus $X$ is
proper as a (path) metric space.

\begin{example}[Hyperbolic group]
The main example of a hyperbolike graph is the Cayley graph of a hyperbolic group
with respect to a finite generating set. Different choices of generating sets
give rise to graphs which are $\delta$-hyperbolike for different $\delta$. Moreover,
the automorphism group of the graph depends on the choice of generating set. We
always have $G \subset \Aut(X)$ where $G$ acts (freely and transitively on the
vertices) on its Cayley graph by left multiplication.
\end{example}

\begin{example}[Free group]
Even for $X$ the Cayley graph of a hyperbolic group $G$, The group $\Aut(X)$ may be
much bigger than $G$. For example, we can take $G$ to be a free group, with
a free generating set. Then $X$ is a regular tree, and $\Aut(X)$ is uncountable.
\end{example}

\begin{example}[Quasi tree with parabolic symmetry group]\label{example:quasitree}
The following example was described to us by Pierre-Emmanuel Caprace. Let $T$ be a
$k$-regular tree (with $k\ge 3$ finite) and fix an end $e$ of $T$. For each vertex
$v$, let $\gamma_v$ denote the geodesic from $v$ to $e$, and let $v'$ denote the
vertex on $\gamma_v$ at distance $2$ from $v$. We obtain $X$ from $T$ by attaching
an edge from each vertex $v$ to the corresponding $v'$. Then $\Aut(X)$ is just the
subgroup of $\Aut(T)$ fixing $e$; in particular, it is vertex transitive, but not
unimodular, therefore there is no discrete subgroup acting cocompactly on $X$.
\end{example}

Pick a base point $x\in X$. Since $\Aut(X)$ is transitive, any two
choices are isomorphic. For any $n$ let $X_n$ denote the ball of
radius $n$ about $x$ in $X$ (i.e.\/ the complete subgraph spanned by
the vertices at distance $\le n$ from $x$).

We let $\Aut(X,x)$ denote the subgroup of $\Aut(X)$ fixing $x$. Evidently, $\Aut(X,x)$
fixes each subgraph $X_n$, so that there are homomorphisms 
$$p_n:\Aut(X,x) \to \Aut(X_n)$$
and we can identify $\Aut(X,x)$ with the inverse limit 
$$\Aut(X,x) = \varprojlim p_n(\Aut(X,x))$$
In particular, $\Aut(X,x)$ is compact, and therefore either finite or uncountable.
In the first case, $\Aut(X)$ is itself finitely generated and a hyperbolic group, and
the orbit map to $X$ is a quasi-isometry. But in general we do not know the answer to
the following:

\begin{question}\label{question:group}
Let $X$ be hyperbolike. Is there a hyperbolic group $G$ quasi-isometric to $X$?
\end{question}

\begin{remark}
If one removes the hypothesis that $X$ be $\delta$-hyperbolic, the analogue of Question~\ref{question:group}
has a {\em negative} answer in general. 

If $T_r$ and $T_s$ are regular trees of valence $r$, $s$ respectively (where $r\ne s$ and
$\infty > r,s > 2$), and if $h_r$, $h_s$ are horofunctions on $T_r$ and $T_s$ respectively, 
the {\em Diestel Leader graph} $DL(r,s)$ is the subgraph of the product $T_r\times T_s$
where $h_r+h_s=0$. These graphs were introduced in \cite{Diestel_Leader}.
Firstly, it was shown in
\cite{Bartholdi_Neuhauser_Woess} that
they do not admit a group action with finitely many orbits and finite vertex stabilizers, and
then it was shown in \cite{Eskin_Fisher_Whyte} that $DL(r,s)$ is not even quasi-isometric
to a Cayley graph.

The Diestel Leader graphs are reminiscent of non-unimodular solvable groups, and it is harder to
imagine an analogue in the hyperbolic world.
\end{remark}

\begin{remark}
Random walks on hyperbolike graphs (and more general graphs with vertex transitive symmetry
groups, which might not be hyperbolic or finite valence) are studied in
\cite{Kaimanovich_Woess}.
\end{remark}

\begin{definition}
Let $Y,Z$ be any two {\em finite} graphs. Let $(Y| Z)$ denote the
number of distinct embeddings of $Y$ as a complete subgraph of $Z$.
For any finite graph $Y$, define the generating function
$$b_Y(t):= \sum (Y|X_n) t^n$$
In words: the coefficients of $b_Y(t)$ count the number of copies of $Y$
in the balls of each fixed radius in $X$.
\end{definition}

By abuse of notation, we can think of $x$ as a graph with 1 vertex, so 
that $b_x(t)$ is the generating function for the sizes of the 
balls $|X_n|$. 

With this notation, we have the formula
$$b_{X_n}(t)/|\Aut(X_n)| = b_x(t)/t^n - \text{polar part at 0}$$ 
To see this, observe first that every embedding of
$X_n$ into $X$ as a complete subgraph (taking $x$ to $x$ without loss of
generality) has image equal to exactly $X_n$. For, every point in $X_n$
is within distance $n$ of $x$, so the image is contained in $X_n$. So
the claim follows by counting.

Then the formula follows, since embeddings of $X_n$ in $X_m$ up to
automorphisms are in bijection with points in $X_{m-n}$.

The following theorem generalizes a result in \cite{Ep2}:

\begin{theorem}[Rationality]\label{theorem:rationality}
Let $X$ be $\delta$-hyperbolike.
For any connected graph $Y$, let $b_Y(t)$ be the generating function
whose coefficient of $t^n$ is the number of distinct embeddings of $Y$ as
a complete subgraph of $X_n$. Then $b_Y(t)$ is rational.
\end{theorem}

The result is known for a Cayley graph of a hyperbolic group, \cite{Ep2}.
\section{Proof of the Rationality Theorem}

In this section we give the proof of Theorem~\ref{theorem:rationality}.
The argument borrows heavily from the well-known proof by Cannon \cite{Ca} in the
case of a hyperbolic group; but there are some subtleties, which are
worth spelling out now in informal language.

The main subtlety is the possibility that there are distinct geodesics
$\gamma$, $\gamma'$ between points $x$ and $y$, and some $\phi \in \Aut(X)$
with $\phi(\gamma)=\gamma'$. This situation can certainly occur: consider
a surface group with a presentation like $\langle a,b,c,d\; | \; [a,b][c,d]\rangle$.
The Cayley graph is the 1-skeleton of the tiling of $\H^2$ by regular octagons with
angles $\pi/4$ at the vertices. Two antipodal vertices of an octagon may be joined
by two distinct paths of length 4 in the Cayley graph, and these paths may
be interchanged by an automorphism of the graph.

This ambiguity makes it tricky to define a regular language of geodesics in bijection
with the elements of $X$. Simply put, there is no way to make such a choice
without breaking the symmetry --- in other words, without finding a subgroup $G$
of $\Aut(X)$ which is still (coarsely) transitive, but acts freely on some rich set
of (sufficiently long) geodesics. Such a subgroup does not exist in general
(e.g.\/ Example~\ref{example:quasitree}), and it is not clear what to use as a substitute;
morally this is the sort of issue we are raising with Question~\ref{question:group}.

\begin{definition}[Synchronous fellow travelers]
Let $\gamma$ and $\gamma'$ be two geodesics with the same initial vertex. Let
their lengths be $\ell$ and $\ell'$ respectively.
For any $T\ge 0$, the geodesics $\gamma$, $\gamma'$ are said to 
{\em $T$-synchronously fellow travel} if
for all $i$ up to $\min(\ell,\ell')$, there is
an inequality
$$d(\gamma(\ell-i),\gamma'(\ell'-i))\le T$$
\end{definition}

\begin{definition}[Competitor]
Let $B_{2\delta+1}(y)$ be the ball of radius $(2\delta+1)$ about $y$ in $X$. 
For any $y$, an element $z \in B_{2\delta+1}(y)$ is a {\em competitor} of $y$
if $d(x,z) \le d(x,y)$, and if {\em some} geodesic from $x$ to $z$ $(2\delta+1)$-synchronously
fellow travels {\em every} geodesic from $x$ to $y$.
\end{definition}

Note that with this definition, $y$ is a competitor of itself, since every geodesic
from $x$ to $y$ $(2\delta+1)$-synchronously fellow travels every (other)
geodesic from $x$ to $y$.

\begin{definition}[Tournament and tournament type]

A function $F$ from
the vertices of $B_{2\delta+1}(y)$ to $\Z$ is a {\em tournament} if it
satisfies the following conditions:
\begin{enumerate}
\item{for any $z$ there is an inequality $d(x,z) - d(x,y) \le F(z) \le d(y,z)$; and}
\item{if $z$ is a competitor of $y$, then $d(x,z) = d(x,y)+F(z)$.}
\end{enumerate}

Two tournaments $F:B_{2\delta+1}(y) \to \Z$ and $F':B_{2\delta +1}(y') \to \Z$
have the same {\em type} if there is an automorphism $\phi \in \Aut(X)$ with
$\phi(y)= \phi(y')$ so that $F = F' \circ \phi$.
\end{definition}

Note that if $F$ is a tournament then $|F(z)|\le 2\delta+1$ for all $z$, so 
there are only finitely many types of tournament.

\begin{remark}\label{remark:tournament}
The meaning of a tournament is roughly as follows.
As we march along a path, we would like to know the relative distance from $x$
to the different elements $z$ in $B_{2\delta+1}(y)$ in order to certify that
we are really traveling along a geodesic. The problem is that it is hard to keep
track of relative distance to points $z$ that are on the periphery. So we keep track
of an {\em upper bound} on their relative distance (i.e.\/ the value $F(z)$),
which measures (roughly) the length of the shortest path from $x$ to $z$ which
stays (synchronously) close to the geodesic we have traveled along.
\end{remark}

\begin{definition}[cone and cone type]
The {\em cone} associated to a point $y$,
denoted $\cone(y)$, is the full
subgraph of $X$ consisting of points $z$ so that $d(x,z)=d(x,y)+d(y,z)$.

We say that $y$ and $y'$ have the same {\em cone-type} if there is
$\phi \in \Aut(X)$ taking $y$ to $y'$ and taking $\cone(y)$ to $\cone(y')$.
\end{definition}

The following lemma is the analogue of Cannon's key lemma, that $(2\delta+1)$-level
determines cone type in hyperbolic groups.

\begin{lemma}[Tournament determines cone type]\label{lemma:Cannon}
Let $X$ be $\delta$-hyperbolike with base point $x$. Let $y$ and $y'$
with tournaments $F$ and $F'$ be given.

Suppose there
is $\phi \in \Aut(X)$ with $F=F'\circ \phi$. Then
$\phi$ takes $\cone(y)$ to $\cone(y')$.
\end{lemma}
\begin{proof}
Let $z \in \cone(y)$. We need to show that $\phi(z) \in \cone(y')$.
This is proved by induction on $d(y,z)$. If $\gamma$ is a geodesic
from $y$ to $z$, and $w$ is the penultimate point on the geodesic,
then $\phi(w) \in \cone(y')$ by the induction hypothesis. 
So if $\phi(z)$ is not in $\cone(y')$ we must have 
$d(x,y')+d(y',\phi(z))\ge d(x,\phi(z))+1$. A geodesic from $x$ to
$\phi(z)$ must pass through $B_{\delta}(y')$ and therefore some
point on that geodesic must be a competitor to $y'$. Applying
$\phi^{-1}$ to the restriction of this geodesic gives a shortcut from
a corresponding competitor to $z$, contrary to the fact that $z$ is in
$\cone(y)$. So $\phi(z) \in \cone(y')$ as claimed.
\end{proof}

\begin{lemma}\label{lemma:tournament}
Let $y \in X$, and let $F:B_{2\delta+1}(y) \to \Z$ be a tournament.
Then for any $y' \in X$ with $d(y,y')=1$ and $d(x,y')=d(x,y)+1$ there is a tournament
$F':B_{2\delta+1}(y') \to \Z$ whose type depends only on the type of $F$
and the choice of $y'$ in the type of $\cone(y)$.
\end{lemma}
\begin{proof}
We construct $F'$ as follows. First, note that $d(x,y')=d(x,y)+1$ means that
there is a geodesic $\gamma$ from $x$ to $y'$ whose penultimate vertex is $y$.

Now, let $z' \in B_{2\delta+1}(y')$ be a competitor of $y'$. Thus, by definition,
there is some geodesic $\gamma'$ from $x$ to $z'$ which $(2\delta+1)$-synchronously
fellow travels $\gamma$.

Let $z$ be on $\gamma'$ with $d(z,z')=1$. Then by the definition of synchronous
fellow-traveling, $d(z,y)\le 2\delta+1$, and $z$ is a competitor of $y$.

So, we define
$$F'(z') = \min  \lbrace F(z)+d(z,z')-1 \; | \; z \in B_{2\delta+1}(y) \rbrace$$
Then $F'(z') \le d(y',z')$ since $F(z) \le d(y,z)$.
Evidently, $d(x,z')=d(x,y')+F'(z')$ for every competitor $z'$ of $y'$. Moreover, for
every $z'$ there is some $z\in B_{2\delta+1}(y)$ with
$$d(x,z') - d(x,y') \le d(x,z) + d(z,z') - d(x,y) - 1 \le F(z) + d(z,z') - 1 = F'(z')$$
(take $z$ to attain the minimum in the definition of $F'(z')$.
If $z'$ is a competitor of $y'$, then each $\le$ becomes $=$).
By definition, the type of $F'$ depends only on the type of $F$ and 
the choice of $y'$ in the type of cone($y$).
\end{proof}

\begin{definition}[child and parent]
A point $y'$ is a {\em child} of $y$, and $y$ is a {\em parent} of $y'$, if
$d(y,y')=1$ and $d(x,y') = d(x,y)+1$.
\end{definition}

Every parent of $y'$ is a competitor of every other parent. Therefore the tournament
type of any parent $y$ determines the number of parents of each child of $y$.

We now give the proof of Theorem~\ref{theorem:rationality}.
\begin{proof}
Define a finite directed graph as follows. Each vertex corresponds to a possible
tournament type. There is a directed edge from the tournament type of
$y,F$ to the tournament type of $y',F'$ if $y'$ is a child of $y$, and $F'$ is
the tournament type constructed from the tournament type of $F$ by 
Lemma~\ref{lemma:tournament}. 
There is a unique vertex, the base vertex, for the tournament type for the base vertex $x$.
We take the connected component of the base vertex in the following argument.

We put a (rational) weight $w$ on the edge from $(y,F)$ to $(y',F')$ which is
equal to the reciprocal of the number of parents of $y'$. We need to show that this
is well-defined; i.e.\/ it can be determined from the tournament type of 
$(y,F)$ and $(y',F')$.
In the ball $B_2(y)$, count the number of $z$ with $F(z)=0$ and $d(z,y')=1$.
We claim that these are in bijection with the parents of $y'$. For, if
$F(z)=0$ then $d(x,z) \le d(x,y)$ by definition, so if furthermore
$d(z,y')=1$ then $d(x,y') \le d(x,z) + 1 \le d(x,y) + 1 = d(x,y')$, so these
inequalities are equalities. Conversely, every parent $z$ is a competitor, and
thus $F(z)=0$ since $d(z,y')=1$. This proves the claim, and shows that the weight
is well-defined.

Label the vertices of the graph by distinct integers, and define a 
non-negative matrix $M$ whose $ij$ entry is equal to $w(e)$ if there is an edge
$e$ from vertex $i$ to vertex $j$, and $0$ otherwise. Let $\iota$ be the row vector
$(1,0,0\cdots 0)$ and let $\1$ be the column vector whose entries are all $1$s.
Then there is a formula
$$b_x(t) = \sum_n (\iota M^n \1) t^n$$
whose coefficients by their form satisfy a finite linear recurrence (due to
the fact that $M$ is a root of its own characteristic polynomial), and therefore
$b_x(t)$ is a rational function.

In \S~\ref{section:definition_statement} we saw that
$b_{X_n}(t)/|\Aut(X_n)|$ is derived from $b_x(t)/t^n$ by throwing away the
polar part at $0$. Thus $b_{X_n}(t)$ is also a rational function.

Finally, for any connected graph $Y$ we choose a base point $y \in Y$ and
an integer $n$ which is at least as big as the diameter of $Y$, and we count how
many copies of $Y$ there are in $X_n$ with $y$ at the center. For each of these
copies, we let $D$ be the least number so that $Y$ is in $X_D$, and call these the
{\em $D$-copies}. From these finitely many coefficients we can reconstruct
$b_Y(t)$ from $b_{X_n}(t)$ in an obvious way and express it as a finite linear
combination of series of the form $b_D(t)$ for $D\le n$.
\end{proof}

\begin{remark}[explanation of the formula]
Let $v_0$ be the base vertex of the directed graph $\Gamma$
in the argument. 
Lemma \ref{lemma:tournament} gives a natural map $B$ from the set
of all finite geodesics starting at  $x$ in $X$ to 
the set of all directed finite paths
starting at  $v_0$ in $\Gamma$. 
Indeed the map $B$ is a bijection. 
The inverse $B^{-1}$ is given by the induction on the length 
of a path. We call the inverse image a {\em lift}. Suppose $v(i_j)$, $j=0,1,2, \cdots, n+1$ 
is a directed path in $\Gamma$  with $v(i_0)=v_0$ and
it is lifted for $0 \le j \le n$ to a geodesic starting at $x$ and ending at $z$ in $X$.
Since there is a directed edge from $v(i_n)$ to $v(i_{n+1})$,
there must be a point $y \in X$ and a child of $y$, $y'$ such that 
the tournament type of $y$ is $v(i_n)$ and the tournament type
of $y'$ is $v(i_{n+1})$, and that there is $\phi\in \Aut(X)$ with 
$\phi(y)=z$. Now extend the geodesic by adding $\phi(y')$
after $z$, which is the lift of $v(i_{n+1})$.
This is a geodesic by Lemma \ref{lemma:Cannon}.

The map $P$ assigning the end points
to those geodesics is a surjection to $X$.
Notice that for each point $y \in X$ with $y\not=x$, by the 
definition of the weight of each edge in $\Gamma$, the total weight of the paths
in the set $B P^{-1}(y)$ is always $1$ (again, by the induction 
on $d(x,y)$). Now the formula follows. 
\end{remark}

\section{Patterson--Sullivan measures for hyperbolike graphs}

From Theorem~\ref{theorem:rationality} and from elementary linear algebra it
follows that if $X$ is $\delta$-hyperbolike for some $\delta$, and is not
quasi-isometric to a point or a line, then there are constants $\lambda>1$
and $C>1$, and an integer $k\ge 0$ so that there is an estimate of the form
$$C^{-1} \lambda^n n^k \le |X_n| \le C \lambda^n n^k$$
In this section we refine this estimate, showing that $k=0$. Explicitly, we show
\begin{theorem}[Exponential]\label{theorem:exponential}
Let $X$ be $\delta$-hyperbolike. Then there are constants $\lambda>1$ and $C>1$ so
that there is an estimate of the form
$$C^{-1} \lambda^n \le |X_n| \le C \lambda^n$$
\end{theorem}

\begin{remark}
If we use the notation $X_{=n}$ for the subset of $X$ at distance {\em exactly} $n$
from the base point, then a similar estimate
$$C^{-1} \lambda^n \le |X_{=n}| \le C \lambda^n$$
holds, with the same constant $\lambda$ but a possibly different constant $C$.
\end{remark}

If $X$ is the Cayley graph of a hyperbolic group, Theorem~\ref{theorem:exponential}
is due to Coornaert \cite{Co}, and is proved by generalizing the theory of Patterson--Sullivan
measures. As explained in \cite{Calegari_ergodic} \S~2.5, the proof of Coornaert's theorem
can be considerably simplified by first showing that the generating function 
$b_x(t)$ is rational, as a corollary of Cannon's theorem for hyperbolic groups.

Our proof of Theorem~\ref{theorem:exponential} runs along very similar lines, and
amounts to little more than the verification that the steps in the argument
given in \cite{Calegari_ergodic} hold in the more general context of hyperbolike graphs.
We carry out this verification in the remainder of the section.

\subsection{Visual boundary}

The first step is to metrize $\partial_\infty X$ following Gromov. Let
$d_X$ denote the ordinary (path) metric in $X$.
\begin{definition}
Fix some base point $x\in X$ and some constant $a>1$. The
{\em $a$-length} of a rectifiable path $\gamma$ in $X$, denoted $\length_a(\gamma)$,
is the integral along $\gamma$ of $a^{-d_X(x,\cdot)}$ with respect to its ordinary 
length, and the {\em $a$-distance} from $y$ to $z$, denoted $d_X^a(y,z)$ is the
infimum of the $a$-lengths of paths between $y$ and $z$.
\end{definition}

The following comparison lemma, due to Gromov, lets us compare $a$-length to ordinary
length.

\begin{lemma}[Gromov]\label{lemma:comparison}
There is some $a_0>1$ so that for $1<a<a_0$ the completion $\overline{X}$ of $X$ in
the $a$-length metric is homeomorphic to $X \cup \partial_\infty X$. Moreover, for
such an $a$ there is a constant $C$ so that for all $y,z \in \partial_\infty X$ there
is an inequality
$$C^{-1}a^{-(y|z)} \le d_X^a(y,z) \le Ca^{-(y|z)}$$
where $(y|z)$ denotes the Gromov product.
\end{lemma}

The Gromov product $(y|z)$ is usually taken to denote the expression
$$(y|z):=\frac 1 2 \left(d_X(x,y) + d_X(x,z) - d_X(y,z)\right)$$
but since we are only ever interested in the value of this expression up to a
(uniformly bounded) additive constant, we could just as easily use the normalization
$(y|z) = d_X(x,yz)$, i.e.\/ the distance from $x$ to some (equivalently, any) geodesic
$yz$ from $y$ to $z$. We stress that this expression is to be interpreted as 
denoting ``equality up to a uniform additive constant''; this (unspecified but effective)
constant will later be absorbed into a multiplicative constant.

\subsection{Patterson--Sullivan measure}

The next step is to construct a (so-called) Patterson--Sullivan (probability) 
measure on $\overline{X}$. Theorem~\ref{theorem:rationality} has the key corollary
that this measure will be supported on the {\em boundary}.

Define the Poincar\'e {\em zeta function} $\zeta_X(s)$, well-defined for $s$ sufficiently
large, by the formula
$$\zeta_X(s):= \sum_{y \in X} e^{-sd_X(x,y)}$$
Recall that we have already shown that there is an estimate of the form
$$C^{-1} \lambda^n n^k \le |X_n| \le C \lambda^n n^k$$
It follows that $\zeta_X$ converges if $s>h:=\log(\lambda)$ and {\em diverges} at $h$.
We may therefore define, for each $s>h$, a probability measure $\nu_s$ on $\overline{X}$
(supported in $X$) by putting an atom of size $e^{-sd_X(x,y)}/\zeta_X(s)$ at each
$y \in X$. Take a subsequence of measures that converges as $s \to h$ from above,
and define $\nu$ to be the limit. By construction, this is a probability measure
{\em supported on $\partial_\infty X$}.

\subsection{Quasiconformal measure}

Recall that if $y \in \partial_\infty X$, a {\em horofunction} $b_y$ centered at
$y$ is a limit of a convergent
subsequence of functions of the form $d_X(y_i,\cdot) - d_X(y_i,x)$
for $y_i \to y$. Such a horofunction is not unique, but is well-defined up to
a uniformly bounded additive constant.

\begin{definition}[Coornaert]
For $\phi \in \Aut(X)$ define $j_\phi:\partial_\infty X \to \R$ by
$$j_\phi(y) = a^{b_y(x) - b_y(\phi(x))}$$
for some horofunction $b_y$ centered at $y$.
A probability measure $\nu$ on $\partial_\infty X$ is {\em quasiconformal of dimension
$D$} if for every $\phi \in \Aut(X)$, the measure $\phi_*\nu$ is absolutely continuous
with respect to $\nu$, and there is a constant $C$ (independent of $\phi$) so that
$$C^{-1} j_\phi(y)^D \le d(\phi_*\nu)/d\nu \le C j_\phi(y)^D$$
\end{definition}
Note that the uniform additive ambiguity in the definition of $b_y$ is absorbed into
a uniform multiplicative ambiguity in the definition of $j_\phi$, which is then
absorbed into the constant $C$; so this definition makes sense.

\begin{proposition}\label{proposition:quasiconformal}
The measure $\nu$ is quasiconformal of dimension $D$, where $D=h/\log{a}$.
\end{proposition}
\begin{proof}
>From the definition of Radon-Nikodym derivative, it suffices to show that there is
a constant $C$, so that for all
$y \in \partial_\infty X$ there is a neighborhood $V$ of $y$ in $\overline{X}$
so that for all $A \subset V$,
$$C^{-1}j_\phi(y)^D \nu(A) \le \nu(\phi^{-1}A) \le Cj_\phi(y)^D\nu(A)$$
By the definition of a horofunction, and $\delta$-thinness, there is a neighborhood
$V$ of $y$ in $\overline{X}$ so that
$$d_X(x,\phi^{-1}z) - d_X(x,z) - C \le b_y(\phi(x)) - b_y(x) \le d_X(x,\phi^{-1}z) - d_X(x,z) + C$$
for some $C$, and for all $z \in V$.

For each $s>h$ we have 
$$\phi_*\nu_s(z)/\nu_s(z) = \nu_s(\phi^{-1}z)/\nu_s(z) = e^{-s(d_X(x,\phi^{-1}z) - d_X(x,z))}$$
Taking $s \to h$ and defining $a^D = e^h$ proves the proposition.
\end{proof}

\subsection{Shadows}

We now recall Sullivan's definition of {\em shadows}:
\begin{definition}
For $y \in X$ and $R>0$ the {\em shadow} $S(y,R)$ is the set of $z \in \partial_\infty X$
such that every geodesic ray from $x$ to $z$ comes within distance $R$ of $y$.
\end{definition}

\begin{lemma}\label{lemma:uniform_cover}
Fix $R>2\delta$. Then there is a constant $N$ so that for any $z\in\partial_\infty X$
and any $n$ there is at least $1$ and there are at most $N$ elements $y$ with
$d_X(x,y)=n$ and $z\in S(y,R)$.
\end{lemma}
\begin{proof}
If $\gamma$ is any geodesic from $x$ to $z$, and if $y$ is any point on $\gamma$, then
$z \in S(y,R)$. Conversely, if $y$ and $y'$ are two elements with $d_X(x,y)=d_X(x,y')$
and $z \in S(y,R)\cap S(y',R)$ then $d_X(y,y')\le 2R$.
\end{proof}

\begin{lemma}\label{lemma:uniform_size}
Fix $R$. Then there is a constant $C$ so that for any $y \in X$ there is an inequality
$$C^{-1} a^{-d_X(x,y)D} \le \nu(S(y,R)) \le Ca^{-d_X(x,y)D}$$
\end{lemma}
\begin{proof}
First observe by $\delta$-thinness and the definition of a shadow, that there is
some constant $C'$ so that
$$d_X(x,y) - C' \le b_z(x) - b_z(y) \le d_X(x,y) + C'$$
for any $z \in S(y,R)$. Since $j_\phi(z) = a^{b_z(x) - b_z(\phi(x))}$ it follows
that there is a constant $C$ so that
$$C^{-1}a^{d_X(x,\phi(x))} \le j_\phi(z) \le Ca^{d_X(x,\phi(x))}$$
for any $\phi \in \Aut(X)$ and any $z \in S(\phi(x),R)$.

Now, since $\nu$ is a quasiconformal measure, $\nu$ cannot consist of a single
atom. So let $m_0<1$ be the measure of the biggest atom of $\nu$, and fix $m_0 < m < 1$.
By compactness of $\partial_\infty X$ there is some $\epsilon$ so that every ball
in $\partial_\infty X$ of diameter $\le \epsilon$ (in the $a$-metric) has mass at most $m$.
Now, for any $\phi \in \Aut(X)$, the set $\phi^{-1}S(\phi(x),R)$ consists of exactly
the $y \in \partial_\infty X$ for which every geodesic ray from $\phi^{-1}(x)$ to $y$
comes within distance $R$ of $x$. As $R \to \infty$, the diameter of
$\partial_\infty X - \phi^{-1}S(\phi(x),R)$ goes to zero uniformly in $\phi$, and
so for some $R_0$, and for all $R\ge R_0$, we have 
$$1-m \le \nu(\phi^{-1} S(\phi(x),R)) \le 1$$
independent of $\phi$.

But by Proposition~\ref{proposition:quasiconformal} and the discussion above,
there is some constant $C_1$ so that
$$C_1a^{d_X(x,\phi(x))D} \le \nu(\phi^{-1}S(\phi(x),R))/\nu(S(\phi(x),R)) \le C_1a^{d_X(x,\phi(x))D}$$
Taking reciprocals, and using $1-m \le \nu(\phi^{-1} S(\phi(x),R)) \le 1$ completes the proof.
\end{proof}

We now give the proof of Theorem~\ref{theorem:exponential}.
\begin{proof}
We already know the lower bound. For each $y$ with $d(x,y)=n$ we have
$e^{-hn} = a^{-Dn} \le C\nu(S(y,R))$. On the other hand, by Lemma~\ref{lemma:uniform_cover},
every point $z \in \partial_\infty X$ is contained in at least 1 and at most $N$ sets
$S(y,R)$ with $d(x,y)=n$. So
$$|X_{=n}|e^{-hn}C^{-1} \le \sum_{d(x,y)=n} \nu(S(y,R)) \le N\nu\Bigl(\bigcup_{d(x,y)=n} S(y,R)\Bigr) = N$$
\end{proof}

\section{The opposite series $\Omega(P)$ for a power series $P$}

\subsection{Definitions}\label{subsection:definitions}

We recall the definition of {\em opposite series} from Saito \cite{Sa}.  
Let $P(t)=\sum_{n=0}^\infty a_nt^n$ be a power series in $t$
with $a_n$ real numbers. Assume there exist $u, v$ such that
for all $n$, $u \le a_{n-1}/a_n \le v$.

Define a polynomial in $s$ for each
$n \ge 0$ as follows: 
$$X_n(P) = \sum_{k=0}^{n} \frac{a_{n-k}}{a_n} s^k.$$
Define $\Omega(P)$ as the set of accumulation points
of the sequence $\{X_n\}_n$ in the set of formal power series on $s$
(with respect to the product topology on each coefficients).
An element in $\Omega(P)$ is called an {\em opposite series}
in \cite[\S 11.2]{Sa}.

Now, let $G$ be a group with a finite generating set $S$.
Let $a_n$ denote the number of elements $g \in G$
whose word length is $n$ with respect to $S$.
Using $a_n$'s, we define $P(t)$, denoted   
by $P_{G,S}$, and obtain $\Omega(P_{G,S})$.

Saito also defined another set $\Omega(G,S)$, 
a map $\pi_{\Omega}:\Omega(G,S) \to \Omega(P_{G,S})$
and proved (Theorem in \S11.2) that the map is surjective under two
assumptions ({\bf S} and {\bf I} in his paper; we will discuss {\bf S}). 
Saito's theory is most interesting when $\Omega(G,S)$ or $\Omega(P_{G,S})$
is finite, but his paper gives only a few examples where finiteness
is shown to hold.

Saito proposed the following conjecture in the last section of his paper 
\cite[\S 12. Conjecture 4]{Sa}:
\begin{conjecture}[Saito] \label{question:saito}
$\Omega(G,S)$ is finite if $G$ is a word hyperbolic group.
\end{conjecture}
In view of Saito's theorem relating $\Omega(G,S)$ to $\Omega(P_{G,S})$, it is natural to ask:
\begin{question}
Is $\Omega(P_{G,S})$  finite  if $G$ is hyperbolic ?
\end{question}

Saito conjectures
that this will be the case \cite{Sa2}, and 
we  will answer this question 
in the affirmative (Corollary \ref{hyp}).
Conjecture \ref{question:saito} is still open. 
Interestingly it turns out that there is an example of a hyperbolic group
which does not satisfy the assumption {\bf S} (see Example
\ref{example:triangle}).

\begin{theorem}[finiteness]\label{main}
Let $X$ be a hyperbolike graph, and for any connected graph $Y$, let $b_Y(t)$
be the generating function whose coefficient of $t^n$ is the number of
distinct embeddings of $Y$ as a complete subgraph of $X_n$. Then $\Omega(b_Y)$
is finite.
\end{theorem}


A special case, which answers Saito's question, is:

\begin{corollary}\label{hyp}
If $G$ is a word hyperbolic group, then 
$\Omega(P_{G,S})$ is finite for any finite generating set $S$.
\end{corollary}

%
%
%
%

\subsection{Proof of Theorem \ref{main} and Corollary \ref{hyp}}
We recall a well-known result from analytic combinatorics. 

\begin{theorem}\cite[Th IV.9]{CUP}\label{cup}
If $f(z)$ is a rational function that is analytic at zero and has poles at points
$\alpha_1, \alpha_2, \cdots \alpha_m$, then its coefficients are a sum of
{\em exponential-polynomials}: there exist $m$ polynomials $\Pi_j(x)$ such that, for
$n$ larger than some fixed $n_0$,
$$f_n = \sum_j \Pi_j(n) \alpha_j^{-n}$$
where $f_n$ is the coefficient of $z^n$ in $f(z)$. Furthermore, the degree of $\Pi_j$
is equal to the order of the pole of $f$ at $\alpha_j$ minus one.
\end{theorem}

Let's apply this to prove Theorem~\ref{main}.
\proof
 The power series $b_Y(t)$ is a rational
function, whose poles are (a subset of) the reciprocals of the roots of the matrix $M$ 
constructed in the proof of Theorem~\ref{theorem:rationality}. Since $M$ is a non-negative
matrix, Perron--Frobenius theory says that there is a root of largest absolute value which
is real and positive, and all other roots with this absolute value differ by multiplication
by a root of unity. From Theorem~\ref{theorem:exponential} and Theorem~\ref{cup} we conclude
that these roots of maximum modulus are {\em simple}, or else the dominant term in the
growth rate of the coefficients of $b_Y(t)$ would be of the form polynomial times exponential,
where the polynomial had positive degree (contrary to Theorem~\ref{theorem:exponential}).

It follows from Theorem~\ref{cup} that for $n$ sufficiently big, there is some $m_0 \le m$
so that after reordering the poles of $b_Y(t)$ in non-decreasing modulus, 
we have an expression of the form
$$f_n = \sum_{j\le m_0} \pi_j \alpha_j^{-n} + \sum_{j>m_0} \Pi_j(n) \alpha_j^{-n}$$
where $\alpha_1$ is real and positive, where $\alpha_j$ for $j\le m_0$ is of the form 
$\alpha_1 \omega_j$ for some root of unity $\omega_j$, and where $|\alpha_j| > \alpha_1$
for $j>m_0$.

Evidently $\alpha_1^{-1} = \lambda$ with notation from Theorem~\ref{theorem:exponential}.
Moreover, if $N$ is the least common multiple of the order of the roots of unity $\omega_j$,
then we can rewrite this expression as
$$f_n = C_{[n]} \lambda^n + o(\lambda^n)$$
where $C_{[n]}$ depends only on the residue of $n$ mod $N$. Again, by Theorem~\ref{theorem:exponential}
we can conclude that $C_{[n]}$ is real and {\em positive} for all $n$ mod $N$.

If we define the polynomial $X_n(b_Y) = \sum_{k=0}^n \frac {f_{n-k}} {f_n} s^k$ as in
Definition~\ref{subsection:definitions}, then as $n \to \infty$ for every fixed $k$ the
coefficient of $s^k$ in $X_n(P)$ approaches a value depending only on $n$ mod $N$. Hence there
are finitely many accumulation points of the $X_n$, which is exactly the conclusion of 
Theorem~\ref{main}.
\qed

\section{Dead ends}\label{section:dead}

\begin{definition}
Let $X$ be a graph and $x$ a base point. A vertex $y$ is a {\em dead end} if there
is no $z \ne y$ with $d(x,z) = d(x,y) + d(y,z)$.
\end{definition}

It is important for Saito to study graphs with the additional hypothesis that
the asymptotic density of dead end elements is zero.
This is one of the assumptions he puts in the main theorem 
in \cite[\S 11.2, Assumption 2. {\bf S}]{Sa}.
 Unfortunately, we show now 
that this hypothesis is genuinely restrictive, since there are (very simple) hyperbolic
groups with finite generating sets whose Cayley graphs have a positive density of
dead ends. Actually, these examples are already well-known; we simply bring them
up to point out the implications for Saito's theory.

The following example is worked out in detail by Pfeiffer \cite{Pf}, Appendix~C; we summarize the
story. 

\begin{example}[Triangle group]\label{example:triangle}
Let $G$ be the $(2,3,7)$ triangle group; i.e.\/ the group with the following presentation
$$G:=\langle a,b \; | \; a^2, b^3, (ab)^7\rangle$$
We abbreviate $b^{-1}$ by $B$.
Every geodesic word in $G$ alternates between $a$ and either $b$ or $B$.

Moreover, {\em infinite} geodesics are exactly those that don't contain 
(except possibly at the very start)
substrings of the form $ababab$ or $aBaBaB$. For, suppose $ababab$ appears in the 
middle of the word. It must be followed by an $a$, and preceded by either $b$ or $B$. 
If we have $babababa$ then of course we can replace it by $aBaBaB$ which is shorter. 
If we have $Babababa$ we can rewrite it as $BBaBaBaB = baBaBaB$ which is shorter.

Now, if $W$ is any word with at most 2 consecutive $ab$s or $aB$s in a row (and is
therefore a geodesic), we can extend it to something like 
$WXBabaBababab$ which now we claim is a dead-end. 
For, it can only be extended to $$WXBabaBabababa = WXBabaBBaBaBaB = WXBababaBaBaB$$
which is definitely shorter. On the other hand, $WXBabaBababab$ is itself a geodesic; 
trying to rewrite it, one can only replace $ababab$ by $BaBaBaBa$ giving 
$$WXBabaBBaBaBaBa = WXBababaBaBaBa$$ which is longer. 

Thus this group has dead end elements with positive density (at least $2^{-6}$).
\end{example}


\begin{thebibliography}{CUP}

\bibitem[BNW]{Bartholdi_Neuhauser_Woess}
Laurent Bartholdi, Markus Neuhauser and Wolfgang Woess,
Horocyclic products of trees,
Jour. Eur. Math. Soc. {\bf 10} (2008), no. 3, 771--816

\bibitem[C]{Calegari_ergodic}
Danny Calegari,
The ergodic theory of hyperbolic groups,
Contemp. Math. {\bf 597} (2013), 15--52.


\bibitem[CF]{CF}
Danny Calegari and Koji Fujiwara,
 Combable functions, quasimorphisms, and the central limit theorem. 
Ergodic Theory Dynam. Systems {\bf 30} (2010), no. 5, 1343–-1369.

\bibitem[Ca]{Ca}
James W. Cannon, 
The combinatorial structure of cocompact discrete hyperbolic groups. Geom. Dedicata {\bf 16} (1984), no. 2, 123–-148. 

\bibitem[Co]{Co}
Michel Coornaert,
Mesures de Patterson-Sullivan sur le bord d'un espace hyperbolique au sens de Gromov. 
Pacific J. Math. {\bf 159} (1993), no. 2, 241–-270. 

\bibitem[DL]{Diestel_Leader}
Reinhard Diestel and Imre Leader,
A conjecture concerning a limit of non-Cayley graphs,
J. Algebraic Combin. {\bf 14} (2001), no. 1, 17--25

\bibitem[Ep]{Ep}
 David B. A. Epstein,  James W. Cannon,  Derek F. Holt, Silvio V. F. Levy, Michael S. Paterson, William P. Thurston,  
(1992), Word Processing in Groups, Boston, MA: Jones and Bartlett Publishers.

\bibitem[Ep2]{Ep2}
David B. A. Epstein,  Anthony R. Iano-Fletcher, Uri Zwick, 
Growth functions and automatic groups. 
Experiment. Math. {\bf 5} (1996), no. 4, 297–-315. 

\bibitem[EFW]{Eskin_Fisher_Whyte}
Alex Eskin, David Fisher and Kevin Whyte,
Coarse differentiation of quasi-isometries I: Spaces not quasi-isometric to Cayley graphs,
Ann. Math. (2) {\bf 176} (2012), no. 1, 221--260

\bibitem[FlSe]{CUP}
Philippe Flajolet, Robert Sedgewick, 
Analytic combinatorics. Cambridge University Press, Cambridge, 2009.

\bibitem[KW]{Kaimanovich_Woess}
Vadim Kaimanovich and Wolfgang Woess,
Boundary and entropy of space homogeneous Markov chains,
Ann. Prob. {\bf 30} (2002), no. 1, 323--363

\bibitem[Pf]{Pf}Markus Pfeiffer.
Automata and Growth Functions for the Triangle Groups.
Diploma Thesis in Computer Science, 
Rheinisch-Westf{\"a}lische Technische Hochschule Aachen.
March 2008.


\bibitem[Sa]{Sa}
Kyoji Saito, 
Limit elements in the configuration algebra for a cancellative monoid. 
Publ. Res. Inst. Math. Sci. {\bf 46} (2010), no. 1, 37–-113. 

\bibitem[Sa2]{Sa2}
 Kyoji Saito,
Opposite power series. European J. Combin. {\bf 33} (2012), no. 7, 1653–-1671. 

\bibitem[Sa3]{Sa3}
Kyoji Saito, 
Private communication. 
\end{thebibliography}
\end{document}